\documentclass[12pt,reqno]{amsart}
\usepackage{amsmath, amsthm, amssymb}

\advance \topmargin by -\headheight
\advance \topmargin by -\headsep
     
\evensidemargin \oddsidemargin
\newtheorem{theorem}{Theorem}[section]
\newtheorem{lemma}[theorem]{Lemma}
\newtheorem{corollary}[theorem]{Corollary}

\theoremstyle{definition}

\theoremstyle{remark}

\numberwithin{equation}{section}

\newcommand{\mmod}[1]{\,\,(\text{mod}\,\,#1)}

  \def\bff{{\mathbf f}}

\def\bfh{{\mathbf h}}

\def\bfm{{\mathbf m}}
\def\bfn{{\mathbf n}}

\def\bfu{{\mathbf u}}

\def\bfw{{\mathbf w}}
\def\bfx{{\mathbf x}}
\def\bfy{{\mathbf y}}
\def\bfz{{\mathbf z}}

\def\calB{{\mathcal B}} 
\def\calC{{\mathcal C}} 

\def\calD{{\mathcal D}}

\def\calN{{\mathcal N}}

\def\calR{{\mathcal R}}

\def\atil{\tilde{a}}\def\btil{\tilde{b}}

\def\dbC{{\mathbb C}}\def\dbN{{\mathbb N}}
\def\dbR{{\mathbb R}}
\def\dbZ{{\mathbb Z}}

\def\gra{{\mathfrak a}}
\def\grb{{\mathfrak b}}

\def\grf{{\mathfrak f}}\def\grF{{\mathfrak F}}

\def\grJ{{\mathfrak J}}

\def\grS{{\mathfrak S}}

\def\alp{{\alpha}} \def\bfalp{{\boldsymbol \alpha}}
\def\bet{{\beta}}  \def\bfbet{{\boldsymbol \beta}}
\def\gam{{\gamma}} 
\def\del{{\delta}} \def\Del{{\Delta}} 
\def\zet{{\zeta}} \def\bfzet{{\boldsymbol \zeta}} 
 
\def\tet{{\theta}}  \def\Tet{{\Theta}}
\def\kap{{\kappa}}
\def\lam{{\lambda}} \def\Lam{{\Lambda}} 

\def\bfxi{{\boldsymbol \xi}}

\def\sig{{\sigma}}

\def\ome{{\omega}} \def\Ome{{\Omega}}
\def\d{{\partial}}
\def\eps{\varepsilon}

\def\le{\leqslant} \def\ge{\geqslant}

\def\d{{\,{\rm d}}}

\def\llbracket{\lbrack\;\!\!\lbrack} \def\rrbracket{\rbrack\;\!\!\rbrack}

\begin{document}
\title[Vinogradov's mean value theorem]{The cubic case of the main conjecture in Vinogradov's mean 
value theorem}
\author[Trevor D. Wooley]{Trevor D. Wooley}
\address{School of Mathematics, University of Bristol, University Walk, Clifton, Bristol BS8 1TW, United 
Kingdom}
\email{matdw@bristol.ac.uk}
\date{}
\begin{abstract} We apply a variant of the multigrade efficient congruencing method to estimate 
Vinogradov's integral of degree $3$ for moments of order $2s$, establishing strongly diagonal behaviour 
for $1\le s\le 6$. Consequently, the main conjecture is now known to hold for the first time in a case of 
degree exceeding $2$.\end{abstract}
\maketitle

\section{Introduction} When $k$ and $s$ are natural numbers, and $X$ is a large real number, denote by 
$J_{s,k}(X)$ the number of integral solutions of the system
\begin{equation}\label{1.1}
x_1^j+\ldots +x_s^j=y_1^j+\ldots +y_s^j\quad (1\le j\le k),
\end{equation}
with $1\le x_i,y_i\le X$ $(1\le i\le s)$. The {\it main conjecture} in Vinogradov's mean value theorem 
asserts that for each $\eps>0$, one has
\begin{equation}\label{1.2}
J_{s,k}(X)\ll X^\eps (X^s+X^{2s-\frac{1}{2}k(k+1)}),
\end{equation}
an estimate that, but for the presence of the factor $X^\eps$, would be best possible (see 
\cite[equation (7.4)]{Vau1997}). Despite eighty years of intense investigation, such an estimate has been 
established only in two cases, namely the (trivial) linear case with $k=1$, and the quadratic case with 
$k=2$ in which the elementary theory of quadratic forms can be brought to bear. Our goal in this paper is 
the first proof of the main conjecture (\ref{1.2}) in a case with $k>2$.

\begin{theorem}\label{theorem1.1} For each $\eps>0$, one has $J_{s,3}(X)\ll X^\eps (X^s+X^{2s-6})$.
\end{theorem}

The estimate for $J_{s,3}(X)$ recorded in this theorem, which establishes the main conjecture in 
Vinogradov's mean value theorem in the cubic case $k=3$, goes substantially beyond the estimates 
available hitherto. By means of Newton's formulae concerning the roots of polynomials, it is apparent that 
$J_{s,3}(X)=s!X^s+O(X^{s-1})$ for $1\le s\le 3$, since the solutions of (\ref{1.1}) are then simply the 
diagonal ones with $\{x_1,\ldots ,x_s\}=\{y_1,\ldots ,y_s\}$. Moreover, from 
\cite[Theorem 1.5]{VW1995} one has
$$J_{4,3}(X)=4!X^4+O(X^{10/3}(\log 2X)^{35}).$$
These estimates confirm (\ref{1.2}) for $1\le s\le 4$ in a particularly strong form when $k=3$, though in 
the latter range the estimate (\ref{1.2}) has been known since at least the time of Hua \cite{Hua1947}. 
Meanwhile, it follows from \cite[Theorem 7]{Hua1947} that when $s\ge 8$, then one has
\begin{equation}\label{1.3}
J_{s,3}(X)\ll X^{2s-6+\eps},
\end{equation}
a conclusion very recently improved in \cite[Corollary 1.2]{Woo2014a} to the extent that (\ref{1.3}) is 
now known to hold for $s\ge 7$. The situations with $s=5$ and $6$ have, however, thus far defied 
resolution.\par

Our strategy for proving Theorem \ref{theorem1.1} is based on the multigrade efficient congruencing 
method introduced in our recent  work \cite{Woo2014a}, and further developed in \cite{Woo2014b}. 
Indeed, the second of these papers shows that, when $k$ is sufficiently large, one has the bound 
$J_{s,k}(X)\ll X^{s+\eps}$ for $1\le s\le \tfrac{1}{2}k(k+1)-\tfrac{1}{3}k+o(k)$, narrowly missing a 
proof of the main conjecture (\ref{1.2}) throughout the critical interval $1\le s\le \tfrac{1}{2}k(k+1)$. A 
careful inspection of the methods underlying the proof of this result shows, however, that these methods 
can be adapted to the case $k=3$, and would narrowly miss a proof of the estimate
\begin{equation}\label{1.4}
J_{6,3}(X)\ll X^{6+\eps}.
\end{equation}
Suitable application of H\"older's inequality in fact leads from such an estimate to the proof of the main 
conjecture in full for $k=3$. In this paper, we are able to devise some modifications to the basic method 
that circumvent these implicit difficulties, leading to a proof of the estimate (\ref{1.4}), and hence the 
proof of Theorem \ref{theorem1.1}. We consequently economise in our exposition by reference to 
\cite{Woo2014b} in several places, though we aim to be transparent where confusion might otherwise 
occur.\par

Our account of the proof of Theorem \ref{theorem1.1} is split up into digestible stages spanning \S\S2--7. 
Aficionados of recent developments concerning Vinogradov's mean value theorem will recognise the basic 
structural features of this plan of attack, although novel elements must be incorporated as we proceed. We 
finish in \S8 by noting a couple of applications of our new estimate. Further applications are available 
associated with the related exponential sums
$$\sum_{1\le x\le X}e(\alp x^3+\bet x)\quad \text{and}\quad \sum_{1\le x\le X}e(\alp x^3+\bet x^2),
$$
where, as usual, we write $e(z)$ for $e^{2\pi iz}$. However, these applications require somewhat 
elaborate arguments that preclude their inclusion in this paper, and so we defer accounts of such 
developments to forthcoming papers \cite{Woo2014c,Woo2014d} elsewhere. The proof of the cubic case 
of the main conjecture seems worthy in its own right as the highlight of this memoir.\par

Finally, we note that a modification of the argument that we engineer here to establish Theorem 
\ref{theorem1.1} can in fact be adapted so as to establish a new bound for $J_{s,k}(X)$ when $k>3$. We 
take this opportunity to announce this new result.

\begin{theorem}\label{theorem1.2}
Suppose that $k\ge 3$ and $s\ge k(k-1)$. Then for each $\eps>0$, one has 
$J_{s,k}(X)\ll X^{2s-\frac{1}{2}k(k+1)+\eps}$.
\end{theorem}

This estimate improves on \cite[Corollary 1.2]{Woo2014a}, where we show that the estimate presented in 
Theorem \ref{theorem1.2} holds for $s\ge k^2-k+1$. Details of the proof of this new estimate will appear 
in a forthcoming paper.

\section{The basic infrastructure} We prepare for the proof of Theorem \ref{theorem1.1} by introducing 
the notation and apparatus required in the iterative method that we ultimately engineer. This is based on 
our recent work \cite{Woo2014b}, though we deviate somewhat in order to circumvent a number of 
technical difficulties. We abbreviate $J_{s,3}(X)$ to $J_s(X)$, and also $J_{6,3}(X)$ to $J(X)$, without 
further comment, and we define $\lam\in \dbR$ by means of the relation
$$\lam=\underset{X\rightarrow \infty}{\lim \sup}\frac{\log J(X)}{\log X}.$$
It follows that for each $\eps>0$, and any $X\in \dbR$ sufficiently large in terms of $\eps$, one has 
$J(X)\ll X^{\lam+\eps}$.\par

Next we recall some standard notational conventions. The letter $\eps$ denotes a sufficiently small positive 
number. Our basic parameter is $X$, a large real number depending at most on $\eps$, unless otherwise 
indicated. Whenever $\eps$ appears in a statement, we assert that the statement holds for each $\eps>0$. 
As usual, we write $\lfloor \psi\rfloor$ to denote the largest integer no larger than $\psi$, and 
$\lceil \psi\rceil$ to denote the least integer no smaller than $\psi$. We make sweeping use of vector 
notation. Thus, with $t$ implied from the ambient environment, we write $\bfz\equiv \bfw\mmod{p}$ to 
denote that $z_i\equiv w_i\mmod{p}$ $(1\le i\le t)$, or $\bfz\equiv \xi\mmod{p}$ to denote that 
$z_i\equiv \xi\mmod{p}$ $(1\le i\le t)$. Finally, we employ the convention that whenever 
$G:[0,1)^3\rightarrow \dbC$ is integrable, then
$$\oint G(\bfalp)\d\bfalp =\int_{[0,1)^3}G(\bfalp)\d\bfalp .$$
Thus, on writing
\begin{equation}\label{2.1}
f(\bfalp;X)=\sum_{1\le x\le X}e(\alp_1x+\alp_2x^2+\alp_3x^3),
\end{equation}
it follows from orthogonality that
\begin{equation}\label{2.2}
J_s(X)=\oint |f(\bfalp;X)|^{2s}\d\bfalp .
\end{equation}

\par We next introduce the parameters appearing in our iterative method. We consider a positive number 
$\Del$ with $12\Del<1$ to be chosen in due course. Put
\begin{equation}\label{2.3}
\gra=\tfrac{2}{3}(7+2\Del)\quad \text{and}\quad \grb=\tfrac{8}{3}(1+\Del),
\end{equation}
and then define
\begin{equation}\label{2.4}
\tet_+=\tfrac{1}{2}(\gra+\sqrt{\gra^2-4\grb})\quad \text{and}\quad \tet_-=
\tfrac{1}{2}(\gra-\sqrt{\gra^2-4\grb}).
\end{equation}
Notice here that
$$\tet_\pm =\tfrac{1}{3}\left( 7+2\Del\pm \sqrt{25+4\Del+4\Del^2}\right) ,$$
so that our choice of $\Del$ ensures that
\begin{equation}\label{2.5}
\tet_+>4+\tfrac{2}{3}\Del\quad \text{and}\quad \tet_-<\tfrac{2}{3}+\tfrac{2}{3}\Del<1.
\end{equation}

\par Our goal is to establish that $\lam\le 6+\Del$. Since we are at liberty to take $\Del$ to be an 
arbitrarily small positive number, it then follows that one has
\begin{equation}\label{2.6}
J_6(X)\ll X^{6+\eps}.
\end{equation}
By applying H\"older's inequality to the right hand side of (\ref{2.2}), we deduce from this estimate that 
whenever $1\le t\le 6$, one has
$$J_t(X)\le \Bigl( \oint |f(\bfalp;X)|^{12}\d\bfalp \Bigr)^{t/6}\ll X^{t+\eps}.$$
Moreover, by applying the trivial estimate $|f(\bfalp;X)|\le P$ in combination with (\ref{2.2}) and 
(\ref{2.6}), we find that when $t>6$, one has
$$J_t(X)\le X^{2t-12}\oint |f(\bfalp;X)|^{12}\d\bfalp \ll X^{2t-6+\eps}.$$
Thus the main conjecture in the cubic case of Vinogradov's mean value theorem does indeed follow from 
(\ref{2.6}).\par

Let $R$ be a natural number sufficiently large in terms of $\Del$. Specifically, we choose $R$ as follows. 
Since $\tet_+>4$, we may put $\nu=\tet_+-4>0$. Then we have
$$4^n=\tet_+^n(1-\nu/\tet_+)^n\le \tet_+^ne^{-\nu n/\tet_+}.$$
Consequently, if we take $R=\lceil W\tet_+/\nu\rceil $, with $W$ a large enough integer, then we ensure 
that
\begin{equation}\label{2.7}
4^R\le e^{-W}\tet_+^R<\frac{\tet_+^{R+1}-\tet_-^{R+1}}{\tet_+-\tet_-}-\tfrac{1}{2}\tet_+\tet_-
\left( \frac{\tet_+^R-\tet_-^R}{\tet_+-\tet_-}\right) .
\end{equation}
The significance of this condition will become apparent in due course (see the discussion surrounding 
(\ref{6.1}) below). Having fixed $R$ satisfying this condition, we take $N$ to be a natural number 
sufficiently large in terms of $R$, and put
\begin{equation}\label{2.8}
B=3^NN,\quad \tet=(200N^2)^{-3RN},\quad \del=(10N)^{-12RN}\tet .
\end{equation}
In view of the definition of $\lam$, there exists a sequence of natural numbers $(X_l)_{l=1}^\infty$, 
tending to infinity with $l$, and with the property that $J(X_l)>X_l^{\lam-\del}$ $(l\in \dbN)$. Also, 
provided that $X_l$ is sufficiently large, one has the corresponding upper bound $J(Y)<Y^{\lam+\del}$ 
for $Y\ge X_l^{1/2}$. We consider a fixed element $X=X_l$ of the sequence $(X_l)_{l=1}^\infty$, which 
we may assume to be sufficiently large in terms of $N$. We put $M=X^\tet$, and note from (\ref{2.8}) 
that $X^\del<M^{1/N}$. Throughout, implicit constants may depend on $N$ and $\eps$, but not on any 
other variable.\par

We next introduce the cast of exponential sums and mean values appearing in our arguments. Let $p$ be a 
prime number with $M<p\le 2M$ to be fixed in due course. When $c$ and $\xi$ are non-negative integers, 
and $\bfalp\in [0,1)^3$, we define
\begin{equation}\label{2.9}
\grf_c(\bfalp;\xi)=\sum_{\substack{1\le x\le X\\ x\equiv
\xi\mmod{p^c}}}e(\alp_1x+\alp_2x^2+\alp_3x^3).
\end{equation}
When $m\in \{1,2\}$, denote by $\Xi_c^m(\xi)$ the set of integral $m$-tuples $(\xi_1,\ldots ,\xi_m)$, 
with $1\le \bfxi\le p^{c+1}$ and $\bfxi\equiv \xi\mmod{p^c}$, and in the case $m=2$ satisfying the 
property that $\xi_1\not\equiv \xi_2\mmod{p^{c+1}}$. We then put
$$\grF_c^m(\bfalp;\xi)=\sum_{\bfxi\in \Xi_c^m(\xi)}\prod_{i=1}^m\grf_{c+1}(\bfalp;\xi_i).$$
Next, when $a$ and $b$ are positive integers, we define
\begin{align*}
I_{a,b}^m(X)&=\max_{1\le \xi\le p^a}\max_{\substack{1\le \eta\le p^b\\ \eta\not\equiv \xi\mmod{p}}}
\oint |\grF_a^m(\bfalp;\xi)^2\grf_b(\bfalp;\eta)^{12-2m}|\d\bfalp ,\\
K_{a,b}^m(X)&=\max_{1\le \xi\le p^a}
\max_{\substack{1\le \eta\le p^b\\ \eta\not\equiv \xi\mmod{p}}}
\oint |\grF_a^m(\bfalp;\xi)^2\grF_b^2(\bfalp;\eta)^2\grf_b(\bfalp;\eta)^{8-2m}|\d\bfalp .
\end{align*}
The implicit dependence on $p$ in the above notation will be rendered irrelevant in \S4, since we fix the 
choice of this prime following Lemma \ref{lemma4.2}.\par

We next align the definition of $K_{a,b}^m(X)$ when $a=0$ with the conditioning idea. When $\xi$ is an 
integer and $\bfzet$ is a tuple of integers, we denote by $\Xi^m(\bfzet)$ the set of $m$-tuples 
$(\xi_1,\ldots ,\xi_m)\in \Xi_0^m(0)$ such that $\xi_i\not\equiv \zet_j\mmod{p}$ for all $i$ and $j$. 
Recalling (\ref{2.9}), we put
$$\grF^m(\bfalp;\bfzet)=\sum_{\bfxi\in \Xi^m(\bfzet)}\prod_{i=1}^m\grf_1(\bfalp;\xi_i),$$
and then define
$$K_{0,c}^m(X)=\max_{1\le \eta\le p^c}\oint |\grF^m(\bfalp;\eta)^2
\grF_c^2(\bfalp;\eta)^2\grf_c(\bfalp;\eta)^{8-2m}|\d\bfalp .$$

\par As in our earlier work, we make use of an operator that indicates the size of a mean value in relation 
to its anticipated magnitude. In the present circumstances, we adopt the convention that
\begin{align}
\llbracket J(X)\rrbracket &=J(X)/X^{6+\Del},\label{2.10}\\
\llbracket I_{a,b}^m(X)\rrbracket &=\frac{I_{a,b}^m(X)}{(X/M^a)^{m+\Del}(X/M^b)^{6-m}},
\label{2.11}\\
\llbracket K_{a,b}^m(X)\rrbracket &=\frac{K_{a,b}^m(X)}{(X/M^a)^{m+\Del}(X/M^b)^{6-m}}.
\label{2.12}
\end{align}
Using this notation, our earlier bounds for $J(X)$ may be written in the form
\begin{equation}\label{2.13}
\llbracket J(X)\rrbracket >X^{\Lam-\del}\quad \text{and}\quad \llbracket 
J(Y)\rrbracket <Y^{\Lam+\del}\quad (Y\ge X^{1/2}),
\end{equation}
where $\Lam$ is defined by $\Lam=\lam-(6+\Del)$.\par

Finally, we recall a simple estimate associated with the system (\ref{1.1}).

\begin{lemma}\label{lemma2.1}
Suppose that $c$ and $d$ are non-negative integers with $c\le \tet^{-1}$ and $d\le \tet^{-1}$. Then 
whenever $u,v\in \dbN$ satisfy $u+v=6$, and $\xi,\zeta\in \dbZ$, one has
$$\oint |\grf_c(\bfalp;\xi)^{2u}\grf_d(\bfalp;\zet)^{2v}|\d\bfalp \ll (J(X/M^c))^{u/6}(J(X/M^d))^{v/6}.
$$
\end{lemma}

\begin{proof} This is immediate from \cite[Corollary 2.2]{FW2013}.
\end{proof}

\section{Auxiliary systems of congruences} We must modify slightly our previous work concerning 
auxiliary congruences so as to accommodate behaviour that deviates slightly from the diagonal. When $a$ 
and $b$ are integers with $1\le a<b$, we denote by $\calB_{a,b}^n(\bfm;\xi,\eta)$ the set of solutions of 
the system of congruences
\begin{equation}\label{3.1}
\sum_{i=1}^n(z_i-\eta)^j\equiv m_j\mmod{p^{jb}}\quad (1\le j\le 3),
\end{equation}
with $1\le \bfz\le p^{3b}$ and $\bfz\equiv \xi\mmod{p^{a+1}}$ for some $\bfxi\in \Xi_a^n(\xi)$. We 
define an equivalence relation $\calR(\lam)$ on integral $n$-tuples by declaring $\bfx$ and $\bfy$ to be 
$\calR(\lam)$-equivalent when $\bfx\equiv \bfy\mmod{p^\lam}$. We then write 
$\calC_{a,b}^{n,h}(\bfm;\xi,\eta)$ for the set of $\calR(hb)$-equivalence classes of 
$\calB_{a,b}^n(\bfm;\xi,\eta)$, and define $B_{a,b}^{n,h}(p)$ by putting 
\begin{equation}\label{3.2}
B_{a,b}^{n,h}(p)=\max_{1\le \xi\le p^a}
\max_{\substack{1\le \eta\le p^b\\ \eta\not\equiv \xi\mmod{p}}}
\max_{1\le \bfm\le p^{3b}}\text{card}(\calC_{a,b}^{n,h}(\bfm;\xi,\eta)).
\end{equation}

\par When $a=0$ we modify these definitions, so that $\calB_{0,b}^n(\bfm;\xi,\eta)$ denotes the set of 
solutions of the system of congruences (\ref{3.1}) with $1\le \bfz\le p^{3b}$ and 
$\bfz\equiv \bfxi\mmod{p}$ for some $\bfxi\in \Xi_0^n(\xi)$, and for which in addition 
$\bfz\not\equiv \eta\mmod{p}$. As in the situation in which one has $a\ge 1$, we write 
$\calC_{0,b}^{n,h}(\bfm;\xi,\eta)$ for the set of $\calR(hb)$-equivalence classes of 
$\calB_{0,b}^n(\bfm;\xi,\eta)$, but we define $B_{0,b}^{n,h}(p)$ by putting
\begin{equation}\label{3.3}
B_{0,b}^{n,h}(p)=\max_{1\le \eta \le p^b}\max_{1\le \bfm\le p^{3b}}
\text{card}(\calC_{0,b}^{n,h}(\bfm;0,\eta)).
\end{equation}

\par We recall a version of Hensel's lemma made available in \cite{Woo1996}.

\begin{lemma}\label{lemma3.1}
Let $f_1,\ldots ,f_d$ be polynomials in $\dbZ[x_1,\ldots ,x_d]$ with respective degrees $k_1,\ldots ,k_d$, 
and write
$$J(\bff;\bfx)=\det\left( \frac{\partial f_j}{\partial x_i}(\bfx)\right)_{1\le i,j\le d}.$$
When $\varpi$ is a prime number, and $l$ is a natural number, let $\calN(\bff;\varpi^l)$ denote the 
number of solutions of the simultaneous congruences
$$f_j(x_1,\ldots ,x_d)\equiv 0\mmod{\varpi^l}\quad (1\le j\le d),$$
with $1\le x_i\le \varpi^l$ $(1\le i\le d)$ and $(J(\bff;\bfx),\varpi)=1$. Then 
$\calN(\bff;\varpi^l)\le k_1\ldots k_d$.
\end{lemma}

\begin{proof} This is \cite[Theorem 1]{Woo1996}.
\end{proof}

We now present the key result on congruences utilised in this paper.

\begin{lemma}\label{lemma3.2}
Suppose that $a$ and $b$ are integers with $0\le a<b$, and that $h$ is a natural number with 
$2b-a\le h\le 2b-a+\Del(b-a)$. Then one has
$$B_{a,b}^{1,3}(p)\le 6\quad \text{and}\quad B_{a,b}^{2,h/b}(p)\le 6p^{h-2b+a}.$$
\end{lemma}

\begin{proof} The estimate $B_{a,b}^{1,3}(p)\le 6$ is immediate from the case $h=3b$, $k=3$ of 
\cite[Lemma 3.1]{Woo2014b}. We therefore focus on establishing the second estimate asserted in the 
statement of the lemma. We begin by considering the situation with $a\ge 1$, the remaining cases with 
$a=0$ being easily accommodated within our argument for the former case. Consider fixed natural 
numbers $a$, $b$ and $h$ with $1\le a\le b$ and
$$2b-a\le h\le 2b-a+\Del (b-a),$$
and fixed integers $\xi$ and $\eta$ with $1\le \xi\le p^a$, $1\le \eta\le p^b$ and 
$\eta \not\equiv \xi\mmod{p}$.
Write $\ome=h-(2b-a)$, so that 
$0\le \ome \le \Del(b-a)$. We denote by $\calD_1(\bfn)$ the set of $\calR(h)$-equivalence classes of 
solutions of the system of congruences
\begin{equation}\label{3.4}
(z_1-\eta)^j+(z_2-\eta)^j\equiv n_j\mmod{p^{2b+\ome}}\quad (j=2,3),
\end{equation}
with $1\le \bfz\le p^{3b}$ and $\bfz\equiv \bfxi\mmod{p^{a+1}}$ for some $\bfxi \in \Xi_a^2(\xi)$. Fix 
an integral triple $\bfm$. To any solution $\bfz$ of (\ref{3.4}) there corresponds a unique pair 
$\bfn=(n_2,n_3)$ with $1\le \bfn\le p^{2b+\ome}$ for which (\ref{3.4}) holds and
$$n_j\equiv m_j\mmod{p^{\sig(j)}}\quad (j=2,3),$$
where $\sig(j)=\min \{ jb,2b+\ome\}$. We therefore infer that
$$\calC_{a,b}^{2,h/b}(\bfm;\xi,\eta)\subseteq \bigcup_{\substack{1\le n_2\le p^{2b+\ome}\\ 
n_2\equiv m_2\mmod{p^{2b}}}}\bigcup_{\substack{1\le n_3\le p^{2b+\ome}\\ 
n_3\equiv m_3\mmod{p^{2b+\ome}}}}\calD_1(\bfn).$$
The number of pairs $\bfn$ in the union is equal to $p^\ome$. Consequently, one has
\begin{equation}\label{3.5}
\text{card}(\calC_{a,b}^{2,h/b}(\bfm;\xi,\eta))\le p^\ome \max_{1\le \bfn\le p^{2b+\ome}}
\text{card}(\calD_1(\bfn)).
\end{equation}

\par Observe that for any solution $\bfz'$ of (\ref{3.4}) there is an $\calR(h)$-equivalent solution $\bfz$ 
satisfying $1\le \bfz\le p^{2b+\ome}$. We next rewrite each variable $z_i$ in the shape 
$z_i=p^ay_i+\xi$. One finds from the hypothesis $\bfz\equiv \bfxi\mmod{p^{a+1}}$ for some 
$\bfxi\in \Xi_a^2(\xi)$ that $y_1\not\equiv y_2\mmod{p}$. Write $\zet=\xi-\eta$, note that 
$p\nmid \zet$, and write the multiplicative inverse of $\zet$ modulo $p^{2b+\ome}$ as $\zet^{-1}$. 
Then we deduce from (\ref{3.4}) that $\text{card}(\calD_1(\bfn))$ is bounded above by the number of 
$\calR(h-a)$-equivalence classes of solutions of the system of congruences
\begin{equation}\label{3.6}
(p^ay_1\zet^{-1}+1)^j+(p^ay_2\zet^{-1}+1)^j\equiv n_j(\zet^{-1})^j\mmod{p^{2b+\ome}}\quad 
(j=2,3),
\end{equation}
with $1\le \bfy\le p^{h-a}$. Recall that $h=2b-a+\ome$, and let $\bfy=\bfw$ be any solution of the 
system (\ref{3.6}), if any one such exists. Then we find that all other solutions $\bfy$ satisfy the system
\begin{equation}\label{3.7}
\sum_{i=1}^2\left( (p^ay_i\zet^{-1}+1)^j-(p^aw_i\zet^{-1}+1)^j\right) 
\equiv 0\mmod{p^{2b+\ome}}\quad (j=2,3).
\end{equation}
When $1\le j\le 3$, write
$$s_j(\bfy,\bfw)=y_1^j+y_2^j-w_1^j-w_2^j.$$
Then by applying the Binomial theorem, it follows that the system (\ref{3.7}) is equivalent to the new system
$$\left. \begin{aligned}
2(\zet^{-1}p^a)s_1(\bfy,\bfw)+&(\zet^{-1}p^a)^2s_2(\bfy,\bfw)&\equiv 0\mmod{p^{2b+\ome}}\\
3(\zet^{-1}p^a)s_1(\bfy,\bfw)+&3(\zet^{-1}p^a)^2s_2(\bfy,\bfw)+
(\zet^{-1}p^a)^3s_3(\bfy,\bfw)&\equiv 0\mmod{p^{2b+\ome}}
\end{aligned}\right\}.$$
By employing the quadratic congruence to eliminate the linear term in the cubic congruence here, one finds 
that this system is in turn equivalent to
$$\left. \begin{aligned}
s_1(\bfy,\bfw)+(2\zet)^{-1}p^as_2(\bfy,\bfw)&\equiv 0\mmod{p^h}\\
s_2(\bfy,\bfw)+2(3\zet)^{-1}p^as_3(\bfy,\bfw)&\equiv 0\mmod{p^{h-a}}
\end{aligned} \right\}.$$

\par Denote by $\calD_2(\bfu)$ the set of $\calR(h-a)$-equivalence classes of solutions of the system of 
congruences
$$\left. \begin{aligned}
y_1+y_2+(2\zet)^{-1}p^a(y_1^2+y_2^2)&\equiv u_2\mmod{p^{h-a}}\\
y_1^2+y_2^2+2(3\zet)^{-1}p^a(y_1^3+y_2^3)&\equiv u_3\mmod{p^{h-a}}
\end{aligned}\right\},$$
with $1\le y_1,y_2\le p^{h-a}$ satisfying $y_1\not \equiv y_2\mmod{p}$. Then we have shown thus far 
that
\begin{equation}\label{3.8}
\text{card}(\calD_1(\bfn))\le \max_{1\le \bfu\le p^{h-a}}\text{card}(\calD_2(\bfu)).
\end{equation}

Next define the determinant
$$J(\bfy)=\det \left( \begin{matrix} 1+2(2\zet)^{-1}p^ay_1&1+2(2\zet)^{-1}p^ay_2\\
2y_1+6(3\zet)^{-1}p^ay_1^2&2y_2+6(3\zet)^{-1}p^ay_2^2\end{matrix} \right) .$$
One has
$$J(\bfy)\equiv 2(y_2-y_1)\not\equiv 0\mmod{p},$$
and hence we deduce from Lemma \ref{lemma3.1} that $\text{card}(\calD_2(\bfu))\le 6$. In combination 
with (\ref{3.5}) and (\ref{3.8}), this estimate delivers the bound
$$\text{card}(\calC_{a,b}^{2,h/b}(\bfm;\xi,\eta))\le 6p^\ome .$$
We thus conclude from (\ref{3.2}) that $B_{a,b}^{n,h/b}(p)\le 6p^{h-2b+a}$, and this completes the 
proof of the lemma when $a\ge 1$.\par

The proof presented above requires little modification to handle the situation in which $a=0$. In this case, 
we denote by $\calD_1(\bfn;\eta)$ the set of solutions of the system of congruences (\ref{3.4}) with 
$1\le \bfz\le p^{3b}$ and $\bfz\equiv \bfxi\mmod{p}$ for some $\bfxi\in \Xi_0^2(0)$, and for 
which in addition $z_i\not\equiv \eta\mmod{p}$ for $i=1,2$. Then as in the opening paragraph of our 
proof, it follows from (\ref{3.4}) that
\begin{equation}\label{3.9}
\text{card}(\calC_{0,b}^{2,h/b}(\bfm;0,\eta))\le p^\ome \max_{1\le \bfn\le p^{2b+\ome}}
\text{card}(\calD_1(\bfn;\eta)).
\end{equation}
But $\text{card}(\calD_1(\bfn;\eta))=\text{card}(\calD_1(\bfn;0))$, and $\text{card}(\calD_1(\bfn;0))$ 
counts the solutions of the system of congruences
$$\left. \begin{aligned}
y_1^3+y_2^3&\equiv n_3\mmod{p^{2b+\ome}}\\
y_1^2+y_2^2&\equiv n_2\mmod{p^{2b+\ome}}
\end{aligned}\right\},$$
with $1\le \bfy\le p^{2b+\ome}$ satisfying $y_1\not\equiv y_2\mmod{p}$ and $p\nmid y_i$ $(i=1,2)$. 
Write
$$J(\bfy)=\det \left( \begin{matrix} 3y_1^2&3y_2^2\\ 2y_1&2y_2\end{matrix} \right) .$$
Then since $p>3$, we have
$$J(\bfy)=6y_1y_2(y_1-y_2)\not\equiv 0\mmod{p}.$$ 
We therefore conclude from Lemma \ref{lemma3.1} that $\text{card}(\calD_1(\bfn;0))\le 6$. In view of 
(\ref{3.3}), the conclusion of the lemma therefore follows from (\ref{3.9}) when $a=0$.
\end{proof}

\section{The conditioning and pre-congruencing processes}
We recall a consequence of a lemma from \cite{Woo2014b} which permits the mean value $I_{a,b}^2(X)$ 
to be bounded in terms of $K_{c,d}^2(X)$, for suitable parameters $c$ and $d$.

\begin{lemma}\label{lemma4.1}
Let $a$ and $b$ be integers with $1\le a<b$, and let $H$ be any integer with $H\ge 15$. Suppose that 
$b+H\le (2\tet)^{-1}$. Then there exists an integer $h$ with $0\le h<H$ having the property that
$$I_{a,b}^2(X)\ll (M^h)^{8/3}K_{a,b+h}^2(X)+M^{-H}(X/M^b)^4(X/M^a)^{\lam-4}.$$
\end{lemma}

\begin{proof} This is simply a special case of \cite[Lemma 4.2]{Woo2014b}.
\end{proof}

Next we recall a lemma from \cite{Woo2014b} which initiates the iterative process.

\begin{lemma}\label{lemma4.2}
There exists a prime number $p$, with $M<p\le 2M$, and an integer $h$ with $0\le h\le 4B$, for which 
one has
$$J(X)\ll M^{8B+8h/3}K_{0,B+h}^2(X).$$
\end{lemma}

\begin{proof} Again, this is simply a special case of \cite[Lemma 5.1]{Woo2014b}.
\end{proof}

We now fix the prime number $p$, once and for all, in accordance with the conclusion of Lemma 
\ref{lemma4.2}.

\section{Efficient congruencing and the multigrade combination}
We adapt the treatment of \cite[\S6]{Woo2014b} to the present cubic situation.

\begin{lemma}\label{lemma5.1} Suppose that $a$ and $b$ are integers with $0\le a<b\le \tet^{-1}$, and 
suppose further that $b\ge (1+\frac{2}{3}\Del)a$. Then one has
\begin{equation}\label{5.1}
K_{a,b}^1(X)\ll M^{3b-a}(I_{b,3b}^2(X))^{1/4}(J(X/M^b))^{3/4}.
\end{equation}
Moreover, whenever $b'$ is an integer with
$$2b-a\le b'\le 2b-a+\Del(b-a),$$
one has
\begin{equation}\label{5.2}
K_{a,b}^2(X)\ll M^{b'-2b+a}(M^{b'-a})^{4/3}(I_{b,b'}^2(X))^{1/3}(K_{a,b}^1(X))^{2/3}.
\end{equation}
\end{lemma}

\begin{proof} The estimate (\ref{5.1}) is the special case $s=4$, $m=0$ of \cite[Lemma 6.1]{Woo2014b} 
corresponding to exponent $k=3$, in which one takes $b'=3b$. We focus, therefore, on the proof of the 
estimate (\ref{5.2}). Even in this situation, however, the argument of the proof of 
\cite[Lemma 6.1]{Woo2014b} applies without serious modification. Applying the latter with $s=4$ and 
$m=1$, we find that the final conclusion must be modified only to reflect the fact that, in view of Lemma 
\ref{lemma3.2}, one has in present circumstances the bound
$$\text{card}(\calC_{a,b}^{2,b'/b}(\bfm;\xi,\eta))\le 6p^{b'-2b+a},$$
whereas in the discussion following \cite[equation (6.5)]{Woo2014b} one had the sharper bound 
$\text{card}(\calC_{a,b}^{2,b'/b}(\bfm;\xi,\eta))\le 6$, owing to the stronger constraint on $b'$ therein. 
On accounting for the presence of the additional factor $p^{b'-2b+a}$ in the analogue of the discussion 
leading from \cite[equation (6.6)]{Woo2014b} to the conclusion of the proof of 
\cite[Lemma 6.1]{Woo2014b}, the upper bound (\ref{5.2}) follows at once. This completes the proof of 
the lemma.
\end{proof}

We note that when $a$ and $b$ are sufficiently large in terms of $\Del$, then the hypothesis 
$b\ge (1+\frac{2}{3}\Del)a$ in the statement of Lemma \ref{lemma5.1} ensures that
\begin{align*}
2b-a+\Del(b-a)&=(2+\Del)b-(1+\Del)a\ge \left( 2+\Del-\frac{1+\Del}{1+\tfrac{2}{3}\Del}\right) b\\
&=\left( 1+\Del - \frac{\tfrac{1}{3}\Del}{1+\tfrac{2}{3}\Del}\right)b\ge 
\lceil (1+\tfrac{2}{3}\Del)b\rceil .
\end{align*}
We are therefore at liberty to apply Lemma \ref{lemma5.1} with a choice for $b'$ satisfying the condition 
$b'\ge (1+\frac{2}{3}\Del)b$, thereby preparing appropriately for subsequent applications of Lemma 
\ref{lemma5.1}.\par

We next combine the estimates supplied by Lemma \ref{lemma5.1} so as to bound $K_{a,b}^2(X)$ in 
terms of the mean values $I_{b,k_mb}^2(X)$ $(m=0,1)$, in which $k_0=3$ and
$$2-a/b\le k_1\le 2-a/b+\Del(1-a/b).$$

\begin{lemma}\label{lemma5.2}
Suppose that $a$ and $b$ are integers with $0\le a<b\le \tet^{-1}$, and suppose further that 
$b\ge (1+\frac{2}{3}\Del)a$. Then whenever $d$ is an integer with $0\le d\le \Del (b-a)$, one has
$$\llbracket K_{a,b}^2(X)\rrbracket \ll \left( (X/M^b)^{\Lam+\del}\right)^{1/2}
\llbracket I_{b,3b}^2(X)\rrbracket^{1/6}\llbracket I_{b,b'}^2(X)\rrbracket^{1/3},$$
where $b'=2b-a+d$.
\end{lemma}

\begin{proof} By substituting the estimate for $K_{a,b}^1(X)$ provided by equation (\ref{5.1}) of Lemma 
\ref{lemma5.1} into (\ref{5.2}), we find that
$$K_{a,b}^2(X)\ll M^d\bigl( (M^{b'-a})^4I_{b,b'}^2(X)\bigr)^{1/3}
\bigl( (M^{3b-a})^4I_{b,3b}^2(X)\bigr)^{1/6}\left( J(X/M^b)\right)^{1/2}.$$
On recalling (\ref{2.10}) to (\ref{2.12}), therefore, we deduce that
$$\llbracket K_{a,b}^2(X)\rrbracket \ll M^\Ome \llbracket I_{b,3b}^2(X)\rrbracket^{1/6} 
\llbracket I_{b,b'}^2(X)\rrbracket^{1/3} \left( (X/M^b)^{\Lam+\del}\right)^{1/2},$$
where
$$\Ome=d+\Del(a-b)\le \Del(b-a)+\Del(a-b)=0.$$
Since $\Ome\le 0$, the conclusion of the lemma is now immediate.
\end{proof}

We next study a multistep multigrade combination stemming from Lemma \ref{lemma5.2}. We begin by 
introducing some additional notation. We recall that $R$ is a positive integer sufficiently large in terms of 
$\Del$. We consider $R$-tuples of integers $(m_1,\ldots ,m_R)\in \{0,1\}^R$, to each of which we 
associate an $R$-tuple of integers $\bfh=(h_1(\bfm),\ldots ,h_R(\bfm))\in [0,\infty)^R$. The integral 
tuples $\bfh(\bfm)$ will be fixed as the iteration proceeds, with $h_n(\bfm)$ depending at most on the 
first $n$ coordinates of $(m_1,\ldots ,m_R)$. We may abuse notation in some circumstances by writing 
$h_n(\bfm,m_n)$ or $h_n(m_1,\ldots ,m_{n-1},m_n)$ in place of $h_n(m_1,\ldots ,m_R)$, reflecting the 
latter implicit dependence. We suppose that a positive integer $b$ has already been fixed. We then define 
the sequences $(a_n)=(a_n(\bfm;\bfh))$ and $(b_n)=(b_n(\bfm;\bfh))$ by putting
\begin{equation}\label{5.3}
a_0=\lfloor b/(1+\tfrac{2}{3}\Del)\rfloor\quad \text{and}\quad b_0=b,
\end{equation}
and then applying the iterative relations, for $1\le n\le R$, given by
\begin{equation}\label{5.4}
a_n=b_{n-1}
\end{equation}
and
\begin{equation}\label{5.5}
b_n=\begin{cases}3b_{n-1}+h_n(\bfm),&\text{when $m_n=0$,}\\
2b_{n-1}-a_{n-1}+\lfloor \Del (b_{n-1}-a_{n-1})\rfloor +h_n(\bfm),&\text{when $m_n=1$.}
\end{cases}
\end{equation}
Next, we define the quantity $\Tet_n(\bfm;\bfh)$ for $0\le n\le R$ by writing
\begin{equation}\label{5.6}
\Tet_n(\bfm;\bfh)=(X/M^b)^{-\Lam-\del}\llbracket K_{a_n,b_n}^2(X)\rrbracket 
+M^{-12\cdot 3^Rb}.
\end{equation}
Finally, we put
$$\phi_0=1/6\quad \text{and}\quad \phi_1=1/3.$$

\begin{lemma}\label{lemma5.3}
Suppose that $a$ and $b$ are integers with $0<a<b\le (16\cdot 3^{2R}R\tet)^{-1}$, and 
suppose further that $a\le b/(1+\tfrac{2}{3}\Del)$. Then there exists a choice for 
$\bfh(\bfm)\in \{0,1\}^R$, satisfying the condition that $0\le h_n(\bfm)\le 15\cdot 3^Rb$ 
$(1\le n\le R)$, and for which one has
$$(X/M^b)^{-\Lam-\del}\llbracket K_{a,b}^2(X)\rrbracket \ll \prod_{\bfm\in\{0,1\}^R}
\Tet_R(\bfm;\bfh)^{\phi_{m_1}\ldots \phi_{m_R}}.$$
\end{lemma}

\begin{proof} A comparison of Lemma \ref{lemma5.2} above with \cite[Lemma 7.2]{Woo2014b} reveals 
that the argument of the proof of \cite[Lemma 7.3]{Woo2014b} applies in the present situation, mutatis 
mutandis, to establish the conclusion of the lemma. We note here that our Lemma \ref{lemma4.1} above 
serves as a substitute for \cite[Lemma 4.2]{Woo2014b} for this purpose.
\end{proof}

\section{The latent monograde process}
We next convert the block estimate encoded in Lemma \ref{lemma5.3} into a single monograde estimate 
that can be incorporated into our iterative method. We begin by recalling an elementary lemma from our 
previous work \cite{Woo2014a}.

\begin{lemma}\label{lemma6.1} Suppose that $z_0,\ldots ,z_l\in \dbC$, and that $\bet_i$ and $\gam_i$ 
are positive real numbers for $0\le i\le l$. Put $\Ome=\bet_0\gam_0+\ldots +\bet_l\gam_l$. Then one 
has
$$|z_0^{\bet_0}\ldots z_l^{\bet_l}|\le \sum_{i=0}^l|z_i|^{\Ome/\gam_i}.$$
\end{lemma}

\begin{proof} This is \cite[Lemma 8.1]{Woo2014a}.
\end{proof}

Before proceeding further, we introduce some additional notation. Define the positive number $s_0$ by 
means of the relation
\begin{equation}\label{6.1}
s_0^R=\frac{\tet_+^{R+1}-\tet_-^{R+1}}{\tet_+-\tet_-}-\frac{\tet_+\tet_-}{2(1+\tfrac{2}{3}\Del)}
\left( \frac{\tet_+^R-\tet_-^R}{\tet_+-\tet_-}\right) ,
\end{equation}
in which $\tet_\pm$ are defined as in (\ref{2.4}). We recall that, in view of (\ref{2.7}), one has $s_0>4$. 
Next we make use of a new pair of sequences $(\atil_n)=(\atil_n(\bfm))$ and $(\btil_n)=(\btil_n(\bfm))$ 
defined by means of the relations
\begin{equation}\label{6.2}
\atil_0=1/(1+\tfrac{2}{3}\Del)\quad \text{and}\quad \btil_0=1,
\end{equation}
and then, when $1\le n\le R$, by
\begin{equation}\label{6.3}
\atil_n=\btil_{n-1}
\end{equation}
and
\begin{equation}\label{6.4}
\btil_n=\begin{cases} 3\btil_{n-1},&\text{when $m_n=0$,}\\
2\btil_{n-1}-\atil_{n-1}+\Del (\btil_{n-1}-\atil_{n-1}),&\text{when $m_n=1$.}
\end{cases}
\end{equation}
We then define
\begin{equation}\label{6.5}
k_\bfm=\btil_R(\bfm)\quad \text{and}\quad \rho_\bfm=\btil_R(\bfm)(4/s_0)^R\quad \text{for}\quad 
\bfm\in \{0,1\}^R.
\end{equation}

\begin{lemma}\label{lemma6.2}
Suppose that $\Lam\ge 0$, let $a$ and $b$ be integers with
$$0\le a<b\le (20\cdot 3^{2R}R\tet)^{-1},$$
and suppose further that $a\le b/(1+\frac{2}{3}\Del)$. Suppose in addition that there are real numbers 
$\psi$, $c$ and $\gam$, with
$$0\le c\le (2\del)^{-1}\tet,\quad \gam\ge -4b\quad \text{and}\quad \psi\ge 0,$$
such that
\begin{equation}\label{6.6}
X^\Lam M^{\Lam \psi}\ll X^{c\del}M^{-\gam}\llbracket K_{a,b}^2(X)\rrbracket .
\end{equation}
Then, for some $\bfm\in \{0,1\}^R$, there is a real number $h$ with $0\le h\le 16\cdot 3^{2R}b$, and 
positive integers $a'$ and $b'$ with $a'\le b'/(1+\frac{2}{3}\Del)$, such that
\begin{equation}\label{6.7}
X^\Lam M^{\Lam \psi'}\ll X^{c'\del}M^{-\gam'}\llbracket K_{a',b'}^2(X)\rrbracket ,
\end{equation}
where $\psi'$, $c'$, $\gam'$ and $b'$ are real numbers satisfying the conditions
$$\psi'=\rho_\bfm (\psi+\tfrac{1}{2}b),\quad c'=\rho_\bfm (c+1),\quad \gam'=\rho_\bfm \gam,\quad 
b'=k_\bfm b+h.$$
Moreover, the real number $k_\bfm$ satisfies $(1+\frac{2}{3}\Del)^R\le k_\bfm\le 3^R$.
\end{lemma}

\begin{proof} We deduce from the postulated bound (\ref{6.6}) and Lemma \ref{lemma5.3} that there 
exists a choice of the tuple $\bfh=\bfh(\bfm)$, with $0\le h_n(\bfm)\le 15\cdot 3^Rb$ $(1\le n\le R)$, 
such that
$$X^\Lam M^{\Lam \psi}\ll X^{(c+1)\del}M^{-\gam}(X/M^b)^\Lam \prod_{\bfm\in \{0,1\}^R}
\Tet_R(\bfm;\bfh)^{\phi_{m_1}\ldots \phi_{m_R}}.$$
Consequently, one has
$$\prod_{\bfm \in \{0,1\}^R}\Tet_R(\bfm;\bfh)^{\phi_{m_1}\ldots \phi_{m_R}}\gg X^{-(c+1)\del}
M^{\Lam (\psi+b)+\gam}.$$
Note that $\phi_0+\phi_1=\tfrac{1}{2}$, so that
$$\sum_{\bfm\in \{0,1\}^R}\phi_{m_1}\ldots \phi_{m_R}=\left( \tfrac{1}{2}\right)^R\le \tfrac{1}{2}.
$$
Then we deduce from the definition (\ref{5.6}) of $\Tet_n(\bfm;\bfh)$ that
\begin{equation}\label{6.8}
\prod_{\bfm\in \{0,1\}^R}\left( X^{-\Lam}\llbracket K_{a_R,b_R}^2(X)\rrbracket 
+M^{-12\cdot 3^R b}\right)^{\phi_{m_1}\ldots \phi_{m_R}}\gg X^{-(c+1)\del}
M^{\Lam (\psi+\frac{1}{2}b)+\gam}.
\end{equation}

\par In preparation for our application of Lemma \ref{lemma6.1}, we examine the exponents 
$\phi_{m_1}\ldots \phi_{m_R}$. Put
$$\bet_\bfm^{(n)}=\phi_{m_1}\ldots \phi_{m_n}\quad \text{and}\quad 
\gam_\bfm^{(n)}=\btil_n(\bfm)\quad (\bfm\in \{0,1\}^n).$$
In addition, we define
$$B_n=\sum_{\bfm\in \{0,1\}^n}\bet_\bfm^{(n)}\btil_n(\bfm)\quad \text{and}\quad 
A_n=\sum_{\bfm\in \{0,1\}^n}\bet_\bfm^{(n)}\atil_n(\bfm),$$
and then put $\Ome=B_R$. From the iterative formulae (\ref{6.2}) to (\ref{6.4}), we obtain
\begin{align*}
B_{n+1}=&\, \frac{1}{6}\sum_{\bfm\in \{0,1\}^n}3\btil_n(\bfm)\phi_{m_1}\ldots \phi_{m_n}\\
&\,+\frac{1}{3}\sum_{\bfm \in \{0,1\}^n}(2\btil_n(\bfm)-\atil_n(\bfm)+
\Del (\btil_n(\bfm)-\atil_n(\bfm))\phi_{m_1}\ldots \phi_{m_n},
\end{align*}
so that
\begin{align*}
B_{n+1}&=\tfrac{1}{2}B_n+(\tfrac{2}{3}+\tfrac{1}{3}\Del)B_n-(\tfrac{1}{3}+\tfrac{1}{3}\Del)A_n\\
=&\,(\tfrac{7}{6}+\tfrac{1}{3}\Del)B_n-(\tfrac{1}{3}+\tfrac{1}{3}\Del)A_n.
\end{align*}
Similarly, one finds that
$$A_{n+1}=\frac{1}{2}\sum_{\bfm\in \{0,1\}^n}\btil_n(\bfm)\phi_{m_1}\ldots \phi_{m_n}=
\tfrac{1}{2}B_n.$$
Thus we conclude via (\ref{2.3}) that
\begin{equation}\label{6.9}
4^2B_{n+2}=\gra (4B_{n+1})-\grb B_n\quad (n\ge 1).
\end{equation}
In addition, one has the initial data
\begin{align}
4B_1&=4\left( \tfrac{1}{6}(3\btil_0)+\tfrac{1}{3}(2\btil_0-\atil_0+\Del(\btil_0-\atil_0))\right) 
=\gra -\tfrac{1}{2}\grb/(1+\tfrac{2}{3}\Del),\label{6.10}\\
4A_1&=4(\tfrac{1}{2}\btil_0)=2,\notag
\end{align}
and hence
\begin{equation}\label{6.11}
4^2B_2=4^2\left( (\tfrac{7}{6}+\tfrac{1}{3}\Del)B_1-\tfrac{1}{3}(1+\Del)A_1\right)
=\gra (\gra-\tfrac{1}{2}\grb/(1+\tfrac{2}{3}\Del))-\grb.
\end{equation}

\par The recurrence formula (\ref{6.9}) has a solution of the shape
$$4^nB_n=\sig_+\tet_+^n+\sig_-\tet_-^n\quad (n\ge 1),$$
where, in view of (\ref{6.10}) and (\ref{6.11}), one has
$$\sig_+\tet_++\sig_-\tet_-=4B_1=\gra-\tfrac{1}{2}\grb/(1+\tfrac{2}{3}\Del)$$
and
$$\sig_+\tet_+^2+\sig_-\tet_-^2=4^2B_2=\gra(\gra-\tfrac{1}{2}\grb/(1+\tfrac{2}{3}\Del))-\grb .$$
Since $\gra=\tet_++\tet_-$ and $\grb=\tet_+\tet_-$, we therefore deduce that
$$4^nB_n=\frac{\tet_+^{n+1}-\tet_-^{n+1}}{\tet_+-\tet_-}-\frac{\tet_+\tet_-}{2(1+\tfrac{2}{3}\Del)}
\left( \frac{\tet_+^n-\tet_-^n}{\tet_+-\tet_-}\right) .$$
In particular, on recalling (\ref{6.1}), we find that $4^RB_R=s_0^R$, so that $B_R=(s_0/4)^R$. Also, 
therefore, it follows from (\ref{2.7}) that $B_R>1$.\par

Returning now to the application of Lemma \ref{lemma6.1}, we note first that $\Ome=B_R$, and hence 
(\ref{6.8}) yields the relation
$$\sum_{\bfm\in \{0,1\}^R}\left( X^{-\Lam}\llbracket K_{a_R,b_R}^2(X)\rrbracket 
+M^{-12\cdot 3^Rb}\right)^{B_R/\btil_R(\bfm)}\gg X^{-(c+1)\del}
M^{\Lam (\psi+\frac{1}{2}b)+\gam}.$$
But in view of (\ref{6.5}), one has $\btil_R(\bfm)/B_R=\rho_\bfm$, and thus we find that for some tuple 
$\bfm\in \{0,1\}^R$, one has
$$X^{-\Lam}\llbracket K_{a_R,b_R}^2(X)\rrbracket +M^{-12\cdot 3^Rb}\gg 
X^{-\rho_\bfm (c+1)\del}M^{\Lam \rho_\bfm (\psi+\frac{1}{2}b)+\rho_\bfm \gam},$$
whence
\begin{equation}\label{6.12}
X^{-\Lam}\llbracket K_{a_R,b_R}^2(X)\rrbracket +M^{-12\cdot 3^Rb}\gg 
X^{-c'\del}M^{\Lam \psi'+\gam'}.
\end{equation}

\par We next remove the term $M^{-12\cdot 3^Rb}$ on the left hand side of (\ref{6.12}). We observe 
that the relations (\ref{6.4}) ensure that $\btil_R(\bfm)\le 3^R$, and hence (\ref{2.7}) and (\ref{6.5}) 
together reveal that $\rho_\bfm\le \btil_R(\bfm)\le 3^R$. By hypothesis, we have $X^{c\del}<M^{1/2}$, 
whence $X^{c'\del}\ll M^{3^R}$. Thus we deduce from (\ref{2.8}) that
$$X^{-c'\del}M^{\Lam \psi'+\gam'}\ge M^{-3^R+\rho_\bfm \gam}\ge M^{-3^R-4\cdot 3^R b}.$$
Since
$$M^{-12\cdot 3^Rb}<M^{-3^R-8\cdot 3^Rb},$$
it follows from (\ref{6.12}) that
\begin{equation}\label{6.13}
X^{-\Lam}\llbracket K_{a_R,b_R}^2(X)\rrbracket \gg X^{-c'\del}M^{\Lam \psi'+\gam'}.
\end{equation} 

\par Our final task consists of extracting appropriate constraints on the parameters $a_R$ and $b_R$. 
Here, a comparison of (\ref{5.3}) to (\ref{5.5}) with (\ref{6.2}) to (\ref{6.4}) reveals that we may follow 
the argument leading from \cite[equation (8.16)]{Woo2014b} to the conclusion of the proof of 
\cite[Lemma 8.2]{Woo2014b}, but substituting $1+\tfrac{2}{3}\Del$ in place of $\sqrt{k}$ throughout. 
The reader should experience little difficulty in adapting the argument given therein to show that
$$k_\bfm b\le b_R\le k_\bfm b+16\cdot 3^{2R}b,$$
and further that
$$a_R=b_{R-1}<b_R/(1+\tfrac{2}{3}\Del).$$
Moreover, one may also verify that $(1+\frac{2}{3}\Del)^R\le k_\bfm\le 3^R$, just as in the conclusion 
of the proof of \cite[Lemma 8.2]{Woo2014b}. The estimate (\ref{6.7}), with all associated conditions, 
therefore follows from (\ref{6.13}) on taking $a'=a_R$ and $b'=b_R$. This completes our account of the 
proof of the lemma.
\end{proof}

\section{The iterative process}
We begin with a crude estimate of use at the conclusion of our argument.

\begin{lemma}\label{lemma7.1}
Suppose that $a$ and $b$ are integers with $0\le a<b\le (2\tet)^{-1}$. Then provided that $\Lam\ge 0$, 
one has
$$\llbracket K_{a,b}^2(X)\rrbracket \ll X^{\Lam+\del}.$$
\end{lemma}

\begin{proof} On considering the underlying Diophantine equations, we deduce from Lemma 
\ref{lemma2.1} that
$$K_{a,b}^2(X)\ll (J(X/M^a))^{1/3}(J(X/M^b))^{2/3},$$
whence
\begin{align*}
\llbracket K_{a,b}^2(X)\rrbracket &\ll 
\frac{X^\del \left( (X/M^a)^{1/3}(X/M^b)^{2/3}\right)^{6+\Del+\Lam}}
{(X/M^a)^{2+\Del}(X/M^b)^4}\\
&\ll X^{\Lam+\del}M^{\frac{2}{3}\Del (a-b)}\ll X^{\Lam+\del}.
\end{align*}
This completes the proof of the lemma.
\end{proof}

We now come to the crescendo of our argument.

\begin{theorem}\label{theorem7.2}
Suppose that $\Del$ is a positive number with $\Del<\tfrac{1}{12}$. Then for each $\eps>0$, one has 
$J(X)\ll X^{6+\Del+\eps}$.
\end{theorem}

\begin{proof} We prove that $\Lam\le 0$, for then the conclusion of the lemma follows at once from 
(\ref{2.13}). Assume then that $\Lam\ge 0$, for otherwise there is nothing to prove. We begin by noting 
that as a consequence of Lemma \ref{lemma4.2}, one finds from (\ref{2.10}) and (\ref{2.12}) that there 
exists an integer $h_{-1}$ with $0\le h_{-1}\le 4B$ such that
$$\llbracket J(X)\rrbracket \ll M^{4B-4h_{-1}/3}\llbracket K_{0,B+h_{-1}}^2(X)\rrbracket .$$
We therefore deduce from (\ref{2.13}) that
\begin{equation}\label{7.1}
X^\Lam \ll X^\del \llbracket J(X)\rrbracket \ll X^\del M^{4B-4h_{-1}/3}\llbracket 
K_{0,B+h_{-1}}^2(X)\rrbracket .
\end{equation}

\par Next we define sequences $(\kap_n)$, $(h_n)$, $(a_n)$, $(b_n)$, $(c_n)$, $(\psi_n)$ and 
$(\gam_n)$, for $0\le n\le N$, in such a way that
\begin{equation}
(1+\tfrac{2}{3}\Del)^R\le \kap_{n-1}\le 3^R,\quad 0\le h_{n-1}\le 16\cdot 3^{2R}b_{n-1}\quad 
(n\ge 1),\label{7.2}
\end{equation}
and
\begin{equation}\label{7.3}
X^\Lam M^{\Lam \psi_n}\ll X^{c_n\del}M^{-\gam_n}\llbracket K_{a_n,b_n}^2(X)\rrbracket .
\end{equation}
We note here that the sequences $(a_n)$ and $(b_n)$ are not directly related to our earlier use of these 
letters. Given a fixed choice for the sequences $(a_n)$, $(\kap_n)$ and $(h_n)$, the remaining sequences 
are defined by means of the relations
\begin{align}
b_{n+1}&=\kap_n b_n+h_n,\label{7.4}\\
c_{n+1}&=(4/s_0)^R\kap_n(c_n+1),\label{7.5}\\
\psi_{n+1}&=(4/s_0)^R\kap_n(\psi_n+\tfrac{1}{2}b_n),\label{7.6}\\
\gam_{n+1}&=(4/s_0)^R\kap_n\gam_n.\label{7.7}
\end{align}
We put
\begin{align*}
\kap_{-1}=3^R,\quad b_{-1}=1,&\quad a_0=0,\quad b_0=B+h_{-1}\\
\psi_0=0,\quad c_0=1,\quad &\gam_0=\tfrac{4}{3}h_{-1}-4B,
\end{align*}
so that both (\ref{7.2}) and (\ref{7.3}) hold with $n=0$ as a consequence of our initial choice of 
$\kap_{-1}$ and $b_{-1}$, together with (\ref{7.1}). We prove by induction that for each non-negative 
integer $n$ with $n<N$, the sequences $(a_m)_{m=0}^n$, $(\kap_m)_{m=0}^n$ and 
$(h_m)_{m=-1}^n$ may be chosen in such a way that
\begin{equation}\label{7.8}
1\le b_n\le \left( 20\cdot 3^{2R}R\tet\right)^{-1},\quad \psi_n\ge 0,\quad \gam_n\ge -4b_n,\quad 
0\le c_n\le (2\del)^{-1}\tet ,
\end{equation}
\begin{equation}\label{7.9}
0\le a_n\le b_n/(1+\tfrac{2}{3}\Del),
\end{equation}
and so that (\ref{7.2}) and (\ref{7.3}) both hold with $n$ replaced by $n+1$.\par

Let $0\le n<N$, and suppose that (\ref{7.2}) and (\ref{7.3}) both hold for the index $n$. We have already 
shown such to be the case for $n=0$. We observe first that from (\ref{7.2}) and (\ref{7.4}), we find that 
$b_n\le 4(17\cdot 3^{2R})^nB$, whence by invoking (\ref{2.8}), we find that for $0\le n\le N$, one has 
$b_n\le (20\cdot 3^{2R}R\tet)^{-1}$. It is apparent from (\ref{7.5}) and (\ref{7.6}) that $c_n$ and 
$\psi_n$ are non-negative for all $n$. Observe also that since $s_0\ge 4$ and $\kap_m\le 3^R$, then by 
iterating (\ref{7.5}) we obtain the bound
\begin{equation}\label{7.10}
c_n\le 3^{Rn}+3^R\left( \frac{3^{Rn}-1}{3^R-1}\right) \le 3^{Rn+1},
\end{equation}
and by reference to (\ref{2.8}), we discern that $c_n\le (2\del)^{-1}\tet$ for $0\le n<N$.\par

In order to bound $\gam_n$, we recall that $s_0\ge 4$ and iterate the relation (\ref{7.7}) to deduce that
\begin{equation}\label{7.11}
\gam_m=(4/s_0)^{Rm}\kap_0\ldots \kap_{m-1}\gam_0\ge-4(4/s_0)^{Rm}\kap_0\ldots \kap_{m-1}B.
\end{equation}
In addition, we find from (\ref{7.4}) that for $m\ge 0$ one has $b_{m+1}\ge \kap_mb_m$, so that an 
inductive argument yields the lower bound
\begin{equation}\label{7.12}
b_m\ge \kap_0\ldots \kap_{m-1}b_0\ge \kap_0\ldots \kap_{m-1}B.
\end{equation}
Hence we deduce from (\ref{7.11}) that $\gam_m\ge -4(4/s_0)^{Rm}b_m>-4b_m$. Assembling this 
conclusion together with those of the previous paragraph, we have shown that (\ref{7.8}) holds for 
$0\le n\le N$.\par

At this point in the argument, we may suppose that (\ref{7.3}), (\ref{7.8}) and (\ref{7.9}) hold for the 
index $n$. An application of Lemma \ref{lemma6.2} therefore reveals that there exist numbers $\kap_n$, 
$h_n$ and $a_n$ satisfying the constraints implied by (\ref{7.2}) with $n$ replaced by $n+1$, for which 
the upper bound (\ref{7.3}) holds for some $a_n$ with $0\le a_n\le b_n/(1+\frac{2}{3}\Del)$, also with 
$n$ replaced by $n+1$. This completes the inductive step, so that in particular (\ref{7.3}) holds for 
$0\le n\le N$.\par

We now exploit the bound just established. Since we have the upper bound 
$b_N\le 4(17\cdot 3^{2R})^N\le (2\tet)^{-1}$, 
it is a consequence of Lemma \ref{lemma7.1} that
$$\llbracket K_{a_N,b_N}^2(X)\rrbracket \ll X^{\Lam+\del}.$$
By combining this with (\ref{7.3}) and (\ref{7.11}), we obtain the bound
\begin{equation}\label{7.13}
X^\Lam M^{\Lam \psi_N}\ll X^{\Lam +(c_N+1)\del}M^{4\kap_0\ldots \kap_{N-1}B(4/s_0)^{RN}}.
\end{equation}
Meanwhile, an application of (\ref{7.10}) in combination with (\ref{2.8}) shows that 
$X^{(c_N+1)\del}<M$. We therefore deduce from (\ref{7.13}) that
$$\Lam \psi_N\le 4(4/s_0)^{RN}\kap_0\ldots \kap_{N-1}B+1.$$
On recalling (\ref{2.5}) and (\ref{6.1}), we see that
$$s_0^R\le \frac{\tet_+^{R+1}}{\tet_+-\tet_-}<
\frac{(4+\tfrac{2}{3}\Del)\tet_+^R}{(4+\tfrac{2}{3}\Del)-(\tfrac{2}{3}+\tfrac{2}{3}\Del)}
<\tfrac{4}{3}\tet_+^R.$$
Thus, since $R$ is sufficiently large, one finds that $s_0<4+2\Del$. Notice here that 
$\kap_n\ge (1+\frac{2}{3}\Del)^R$ and
$$4/s_0\ge 4/(4+2\Del)=1/(1+\tfrac{1}{2}\Del).$$
Hence we deduce that
$$4(4/s_0)^{RN}\kap_0\ldots \kap_{N-1}B\ge 4
\left( \frac{1+\tfrac{2}{3}\Del}{1+\tfrac{1}{2}\Del}\right)^{RN}B\ge 1,$$
so that
\begin{equation}\label{7.14}
\Lam \psi_N\le 9(4/s_0)^{RN}\kap_0\ldots \kap_{N-1}B.
\end{equation}

\par A further application of the lower bound $b_n\ge \kap_0\ldots \kap_{n-1}B$, available from 
(\ref{7.12}), leads from (\ref{7.6}) and the bound $s_0\ge 4$ to the relation
\begin{align*}
\psi_{n+1}&=(4/s_0)^R(\kap_n\psi_n+\tfrac{1}{2}\kap_nb_n)\\
&\ge (4/s_0)^R\kap_n\psi_n+\tfrac{1}{2}(4/s_0)^R\kap_0\ldots \kap_nB\\
&\ge (4/s_0)^R\kap_n\psi_n+\tfrac{1}{2}(4/s_0)^{R(n+1)}\kap_0\ldots \kap_nB.
\end{align*}
An inductive argument therefore delivers the lower bound
$$\psi_N\ge \tfrac{1}{2}N(4/s_0)^{RN}\kap_0\ldots \kap_{N-1}B.$$
Thus we deduce from (\ref{7.14}) that
$$\Lam \le 
\frac{9(4/s_0)^{RN}\kap_0\ldots \kap_{N-1}B}{\tfrac{1}{2}N(4/s_0)^{RN}\kap_0\ldots \kap_{N-1}B}
=\frac{18}{N}.$$
Since we are at liberty to take $N$ as large as we please in terms of $\Del$, we are forced to conclude that 
$\Lam\le 0$. In view of our opening discussion, this completes the proof of the theorem.
\end{proof}

\begin{corollary}\label{corollary7.3}
For each $\eps>0$, one has $J(X)\ll X^{6+\eps}$.
\end{corollary}

\begin{proof} We apply Theorem \ref{theorem7.2} with $\Del=\tfrac{1}{2}\eps$. Then for each 
$\eps'>0$, one has
$$J(X)\ll X^{6+\frac{1}{2}\eps+\eps'},$$
and the desired conclusion follows by taking $\eps'=\tfrac{1}{2}\eps$.
\end{proof}

As we discussed following (\ref{2.6}) above, the conclusion of Corollary \ref{corollary7.3} establishes the 
main conjecture in full for $J_{s,3}(X)$, and thus the proof of Theorem \ref{theorem1.1} is complete.

\section{Applications} We take the opportunity to report on some immediate applications of Theorem 
\ref{theorem1.1}, with brief notes on the necessary arguments. In all cases, the methods of proof are 
standard for those with a passing familiarity with the area, the hard work having been accomplished with 
the proof of Theorem \ref{theorem1.1}.\par

We begin by discussing the anticipated asymptotic formula for $J_s(X)$. Define the singular series
$$\grS_s=\sum_{q=1}^\infty \underset{(q,a_1,a_2,a_3)=1}
{\sum_{a_1=1}^q\sum_{a_2=1}^q\sum_{a_3=1}^q}
\Bigl| q^{-1}\sum_{r=1}^qe((a_1r+a_2r^2+a_3r^3)/q)\Bigr|^{2s},$$
and the singular integral
$$\grJ_s=\int_{\dbR^3}\Bigl| \int_0^1 e(\bet_1\gam+\bet_2\gam^2+\bet_3\gam^3)\d\gam 
\Bigr|^{2s}\d\bfbet .$$

\begin{theorem}\label{theorem8.1}
When $s\ge 7$, one has $J_s(X)\sim \grS_s\grJ_sX^{2s-6}$.
\end{theorem}

\begin{proof} On recalling (\ref{2.1}), it follows from orthogonality that the bound presented in 
Theorem \ref{theorem1.1} delivers the estimate
$$\oint |f(\bfalp;X)|^{12}\d\bfalp \ll X^{6+\eps},$$
we find that the argument of the proof of \cite[Theorem 1.2]{Woo2012a} detailed in 
\cite[\S9]{Woo2012a} applies without modification to establish the claimed asymptotic formula.
\end{proof}

We note that the elementary lower bound $J_s(X)\gg X^{2s-6}$ (see 
\cite[equation (7.4)]{Vau1997}), suffices to confirm that $\grS_s>0$ and $\grJ_s>0$, since one 
has also the estimates $\grS_s\ll 1$ and $\grJ_s\ll 1$ for $s\ge 7$.\par

For comparison, the methods of \cite[Chapter V]{Hua1947} and \cite[Chapter 7]{Vau1997} would 
combine to yield a conclusion analogous to Theorem \ref{theorem8.1}, but subject to the hypothesis 
$s\ge 9$. Our recent work \cite[Corollary 1.2]{Woo2014a} would permit this condition to be sharpened 
slightly to $s\ge 8$. Meanwhile, one may conjecture that for $1\le s\le 5$, one should have 
$J_s(X)\sim s!X^s$. Such is known for $1\le s\le 4$ (see especially \cite{VW1995}), but remains unproven 
for $s=5$. The remaining even moment would be expected to satisfy a different asymptotic formula. Here, 
the philosophy underlying \cite[Appendix]{VW1995} would suggest that $J_6(X)\sim CX^6$, 
with $C=6!+\grS_6\grJ_6$, this corresponding to a sum of the anticipated major arc contribution together 
with the solutions on linear spaces accounted for by the expected minor arc contribution. This seems 
presently to be far beyond our reach. Perhaps it is worth emphasising in this context that one has
$$0<\grS_6\ll 1\quad \text{and}\quad 0<\grJ_6\ll 1.$$
The second of these estimates is plain from the standard theory. For the first, one should use the 
quasi-multiplicative property of
$$\sum_{r=1}^qe((a_1r+a_2r^2+a_3r^3)/q)$$
in order to divide the problem into a consideration of the situation where $q$ is a prime $p$, or a prime 
power $p^h$ with $h\ge 2$. In the latter case, standard estimates (see the proof of 
\cite[Theorem 7.1]{Vau1997}) show that
$$\underset{(p^h,a_1,a_2,a_3)=1}
{\sum_{a_1=1}^{p^h}\sum_{a_2=1}^{p^h}\sum_{a_3=1}^{p^h}}\Bigl| 
p^{-h}\sum_{r=1}^{p^h}e((a_1r+a_2r^2+a_3r^3)/p^h)\Bigr|^{12}\ll p^{3h}(p^{-h/3})^{12}\ll 
p^{-h}.$$
Meanwhile, when $h=1$, one finds from \cite{Weil1948} that
$$p^{-1}\sum_{r=1}^pe((a_1r+a_2r^2+a_3r^3)/p)\ll p^{-1/2}(p,a_1,a_2,a_3)^{1/2},$$
whence
$$\underset{(p,a_1,a_2,a_3)=1}{\sum_{a_1=1}^p\sum_{a_2=1}^p\sum_{a_3=1}^p}
\Bigl| p^{-1}\sum_{r=1}^pe((a_1r+a_2r^2+a_3r^3)/p)\Bigr|^{12}\ll p^{-3}.$$
Thus we deduce that for a suitable fixed $A>0$ one has
$$\grS_6\ll \prod_p(1+Ap^{-2})\ll 1.$$

\par Finally, we consider a diagonal Diophantine system consisting of a cubic, quadratic and linear 
equation. When $s$ is a natural number, and $a_{ij}$ are integers for $1\le i\le 3$ and $1\le j\le s$, we 
write
$$\phi_i(\bfx)=\sum_{j=1}^sa_{ij}x_j^i\quad (1\le i\le 3),$$
and we consider the Diophantine system
\begin{equation}\label{8.1}
\phi_i(\bfx)=0\quad (1\le i\le 3).
\end{equation}
We write $N(B)$ for the number of integral solutions of the system (\ref{8.1}) with $|\bfx|\le B$. We next 
define the (formal) real and $p$-adic densities associated with the system (\ref{8.1}), following Schmidt 
\cite{Sch1985}. When $L>0$, define
$$\lam_L(\eta)=\begin{cases}L(1-L|\eta|),&\text{when $|\eta|\le L^{-1}$,}\\
0,&\text{otherwise.}
\end{cases}$$
We then put
$$\mu_L=\int_{|\bfxi|\le 1}\prod_{i=1}^3\lam_L(\phi_i(\bfxi))\d\bfxi .$$
The limit $\sig_\infty =\underset{L\rightarrow \infty}{\lim}\mu_L$, when it exists, is called the {\it real 
density}. Meanwhile, given a natural number $q$, we write
$$M(q)=\text{card}\{ \bfx\in (\dbZ/q\dbZ)^s:\phi_i(\bfx)\equiv 0\mmod{q}\ (1\le i\le 3)\}.$$
For each prime number $p$, we then put
$$\sig_p=\lim_{H\rightarrow \infty}p^{H(3-s)}M(p^H),$$
provided that this limit exists, and we refer to $\sig_p$ as the {\it $p$-adic density}.

\begin{theorem}\label{theorem8.2}
Let $s$ be a natural number with $s\ge 13$. Suppose that $a_{ij}$ $(1\le i\le 3,\, 1\le j\le s)$ are 
non-zero integers. Suppose, in addition, that the system of equations (\ref{8.1}) possess non-singular real 
and $p$-adic solutions for each prime number $p$. Then one has
$$N(B)\sim \sig_\infty \Bigl(\prod_p\sig_p\Bigr) B^{s-6}.$$
In particular, the system (\ref{8.1}) satisfies the Hasse principle.
\end{theorem}

The argument of the proof here is essentially standard, mirroring that of the proof of Theorem 
\ref{theorem8.1}, and we therefore offer no details. Here, the work of \cite[Chapter V]{Hua1947} 
combines with the methods of \cite[Chapter 7]{Vau1997} to deliver such a conclusion for $s\ge 17$. Our 
present work, in which we require only $s\ge 13$, achieves the limit imposed by the convexity barrier in 
this problem (see \cite{BW2014}). The latter is a practical requirement in applications of the circle method 
for higher degree problems imposed by square-root cancellation considerations for exponential sums, and 
in this instance requires the number of variables $s$ to exceed twice the sum of degrees in the problem.

\bibliographystyle{amsbracket}
\providecommand{\bysame}{\leavevmode\hbox to3em{\hrulefill}\thinspace}

\end{document}